\documentclass[12pt]{article}%
\usepackage{amsfonts}
\usepackage{fancyhdr}
\usepackage{comment}
\usepackage[a4paper, top=2.5cm, bottom=2.5cm, left=2.2cm, right=2.2cm]%
{geometry}
\usepackage{times}
\usepackage{color}
\usepackage{amsmath,amsmath,amssymb}
\usepackage{changepage}
\usepackage{amssymb}
\usepackage{graphicx}%
\usepackage{graphicx,amssymb,enumerate,hyperref,subcaption,graphicx,amsthm,amsmath,amssymb}
\usepackage{bbm}
\usepackage{epstopdf}
\usepackage{algorithm}
\usepackage{rotating}
\usepackage {tikz}
\usetikzlibrary {positioning}

\newtheorem{theorem}{Theorem}

\newtheorem{corollary}[theorem]{Corollary}

\newtheorem{definition}[theorem]{Definition}

\newtheorem{lemma}[theorem]{Lemma}

\newtheorem{remark}[theorem]{Remark}

\usepackage{accents}
\newlength{\dhatheight}

     \def\EE{\mathbb{E}}

     \def\PP{\mathbb{P}}
     
     \def\RR{\mathbb{R}}

\def\calN{{\cal  N}}

\newdimen\biblioindent    \biblioindent=30pt

 at 8truept
\usepackage{mathtools}

\DeclarePairedDelimiter\floor{\lfloor}{\rfloor}

\numberwithin{equation}{section}

\begin{document}

\title{Non-asymptotic bounds for percentiles of independent non-identical random variables}


\author{Dong Xia\\ Hong Kong University of Science and Technology}




\date{(\today)}

\maketitle

\begin{abstract}
This  note displays an interesting phenomenon for the percentiles of independent but non-identical random variables. 
Let $X_1,\cdots,X_n$ be independent random variables obeying non-identical continuous distributions and $X^{(1)}\geq \cdots\geq X^{(n)}$ be the corresponding order statistics. For $p\in(0,1)$, we investigate the $100(1-p)$\%-th percentile $X^{(\floor{pn})}$ and prove the non-asymptotic bounds for $X^{(\floor{pn})}$. In particular, for a wide class of distributions, we discover an intriguing connection between their median and the harmonic mean of the associated standard deviations. 
For example, if $X_k\sim\calN(0,\sigma_k^2)$ for $k=1,\cdots,n$ and $p=\frac{1}{2}$, we show that its median $\big|{\rm Med}\big(X_1,\cdots,X_n\big)\big|= O_P\Big(n^{1/2}\cdot\big(\sum_{k=1}^n\sigma_k^{-1}\big)^{-1}\Big)$ as long as $\{\sigma_k\}_{k=1}^n$ satisfy certain mild non-dispersion property. 
\end{abstract}

\section{Introduction}\label{sec:intro}
Order statistics, rank statistics and sample percentiles like median are prevalent in robust statistics and non-parametric statistical inference. See \cite{boucheron2012concentration}, \cite{latala2010order}, \cite{david2004order}, \cite{arnold1992first}, \cite{reiss2012approximate}, \cite{gibbons1993nonparametric} and references therein. Historically, the order statistics and sample percentiles have been thoroughly investigated for statistical samples, i.e., independent data sampled from identical distributions. In particular, suppose that $X_1,\cdots,X_n$ are i.i.d. random variables with the common cumulative distribution function $F(x)$. For any $p\in(0,1)$, we denote by $\theta_p$ the $100(1-p)$\%-th quantile of $F(x)$, i.e., $\theta_p=\inf\{x: F(x)\geq 1-p\}$. Similarly, we denote by $\hat\theta_p$ the sample $100(1-p)$\%-th percentile of $\{X_k\}_{k=1}^n$, namely, 
$$
\hat\theta_p:=\inf\bigg\{x: \frac{\#\{X_k\leq x: 1\leq k\leq n\}}{n}\geq 1-p \bigg\}.
$$
By Bahadur's representation (\cite[Theorem~5.11]{Shao_2003_book}, \cite{kiefer1967bahadur}, \cite{ghosh1971new}, \cite{bahadur1966note}), it is well-known that
\begin{equation}\label{eq:hat_thetap}
\big|\hat\theta_p-\theta_p \big|=O_P\bigg(\frac{\sqrt{p(1-p)}}{n^{1/2} F'(\theta_p)}\bigg)
\end{equation}
as long as $F'(\theta_p)>0$ exists. If $F(x)$ is symmetric on $\RR$ such that $F(0)=\frac{1}{2}$, then 
$$
\big|{\rm Med}(X_1,\cdots,X_n)\big|=O_P\bigg(\frac{1}{2n^{1/2}F'(0)}\bigg).
$$

Heterogenous noise arises naturally in diverse fields and statistical applications such as sparse model selection (\cite{cavalier2014sparse}) and inverse problems (\cite{johnstone1997wavelet}, \cite{cavalier2002oracle}, \cite{donoho1995nonlinear}). 
 However, compared with the i.i.d. scenario, the sample percentiles of a large and heterogeneous dataset is much less studied. We observe some intriguing phenomenon  for the percentiles of independent but non-identical random variables. 
 
 We begin with an illustrating example. Let $X_1,\cdots,X_{n_1}\sim\calN(0,\sigma_1^2)$ be i.i.d. normal random variables. By (\ref{eq:hat_thetap}) (see also \cite[Theorem~5.11]{Shao_2003_book}), we obtain 
 $$
 \EE \big|{\rm Med}\big(X_1,\cdots,X_{n_1}\big) \big|\approx \frac{\sigma_1}{\sqrt{n_1}}
 $$
 Suppose that another independent dataset $Z_1,\cdots, Z_{n_2}\sim\calN(0,\sigma_2^2)$ with $\sigma_1\gg \sigma_2$ is available,  we are interested in the magnitude of the median of the combined dataset. Obviously, if $n_2=1$, we expect the following inequality
 $$
 \EE \big|{\rm Med}\big(\{X_k\}_{k=1}^{n_1}, Z_1\big) \big|\leq  \EE \big|{\rm Med}\big(\{X_k\}_{k=1}^{n_1}\big) \big|
 $$
 to hold and $\EE \big|{\rm Med}\big(\{X_k\}_{k=1}^{n_1}, Z_1\big) \big|$ should still have size in the order of $\frac{\sigma_1}{\sqrt{n_1}}$ rather than $\sigma_2$ (standard deviation of $Z_1$).  The following question is of our interest. 
 
 {\it Question: for fixed $\sigma_1, n_1$ and $\sigma_2\ll \sigma_1$, how large should $n_2$ be such that we can expect $\EE \big|{\rm Med}\big(\{X_k\}_{k=1}^{n_1}, \{Z_k\}_{k=1}^{n_2}\big) \big|\approx \sigma_2$ ?}

Before presenting the formal answer, we display the simulation result in Figure~\ref{fig:med_normal}. It shows the simulated  $\EE \big|{\rm Med}\big(\{X_k\}_{k=1}^{n_1}, \{Z_k\}_{k=1}^{n_2}\big) \big|$ with respect to $n_2$ for $n_1=80$ and $\sigma_1=1000, \sigma_2=1$ , based on $5000$ repetitions for each $n_2$. Interestingly, we observe that  $\EE \big|{\rm Med}\big(\{X_k\}_{k=1}^{n_1}, \{Z_k\}_{k=1}^{n_2}\big) \big|\approx 1$ when $n_2\approx 18$ which is much smaller than $n_1$. Actually in Section~\ref{sec:normal}, we will prove that 
$$
\big|{\rm Med}\big(X_1,\cdots,X_n\big) \big|=O_P\bigg(n^{1/2}\cdot\Big(\sum_{k=1}^n \sigma_k^{-1}\Big)^{-1}\bigg)
$$
, if $X_k\sim\calN(0,\sigma_k^2)$ are independent. In the aforementioned example, if $\sigma_1\gg \sigma_2$ such that $\frac{n_1}{\sigma_1}\lesssim \frac{n_2}{\sigma_2}$, then we get
$$
\EE\big|{\rm Med}\big(\{X_k\}_{k=1}^{n_1}, \{Z_k\}_{k=1}^{n_2}\big) \big|=O_P\Big(\sigma_2\cdot\frac{n^{1/2}}{n_2}\Big)
$$
which is of the order $\sigma_2$ as long as $n_2\gtrsim \sqrt{n_1}$.  Put it differently, the median of a completely corrupted dataset can be strikingly ameliorated with only a small number of regular data. 
\begin{figure}
\centering
\includegraphics[scale=0.5]{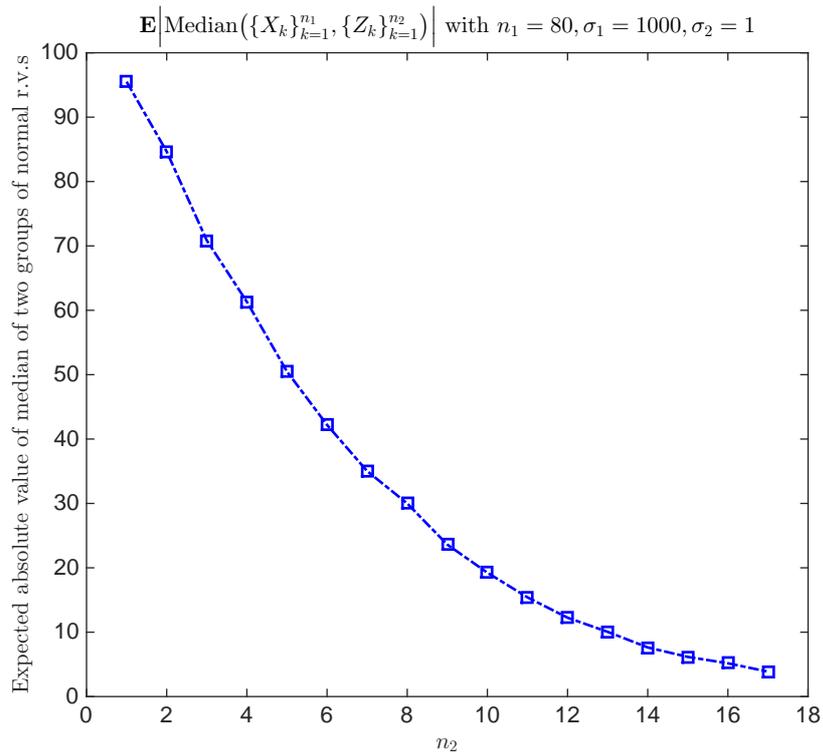}
\caption{Simulation of $\EE \big|{\rm Med}\big(\{X_k\}_{k=1}^{n_1}, \{Z_k\}_{k=1}^{n_2}\big) \big|$ with $X_k\sim\calN(0,\sigma_1^2)$ and $Z_k\sim\calN(0,\sigma_2^2)$ where $n_1=80$, $\sigma_1=1000$ and $\sigma_2=1$. The expectation is computed by the average of $5000$ repetitions for each $n_2$. It shows that the magnitude of the median decreases extremely fast. When $n_2\approx 16$, $\EE \big|{\rm Med}\big(\{X_k\}_{k=1}^{n_1}, \{Z_k\}_{k=1}^{n_2}\big) \big|\approx 5$.}
\label{fig:med_normal}
\end{figure}

The rest of the note will be organized as follows. 
In Section~\ref{sec:con_general}, we develop the concentration inequality for the sample  percentiles of non-identical random variables with general continuous distributions. The proof is simply based on the connection between the concentration of percentiles and the concentration of Bernoulli random variables. We apply those inequalities in Section~\ref{sec:scale-family} to obtain the bounds for the median of non-identical random variables from the scale family distributions including the normal distribution, Cauchy distribution and Laplace distribution. 
Bounds for the general percentiles around the median of normal random variables are presented in Section~\ref{sec:normal_general}.

\section{Concentration inequality of percentiles for general continuous distributions}\label{sec:con_general}
We denote the cumulative distribution function of $X_k$ by $F_k(x)$, i.e., 
$$
F_k(x)=\PP\big(X_k\leq x\big)\quad \textrm{for all } k=1,\cdots,n.
$$
We assume that $F_k(\cdot)$ is continuous in its support for all $1\leq k\leq n$. For all $t\in\RR$, we define
$$
F_N(t)=\frac{1}{n}\cdot\sum_{k=1}^n F_k(t) \in [0,1]
$$
and its corresponding inverse function
$$
F_N^{-1}(p):=\inf\big\{x\in\RR:\ F_N(x)\geq p \big\},\quad \forall\ p\in(0,1).
$$
\begin{lemma}\label{lem:Bern}
For any $t\in\RR$ and $p\in(0,1)$, let $B_1,\cdots,B_n$ denote independent Bernoulli random variables with
$$
\PP(B_k=1)=1-F_k(t)\quad  \textrm{for all } k=1,\cdots,n.
$$
Then,
\begin{equation}\label{eq:Bern_1}
\PP\big(X^{(\floor{pn})}\geq t\big)=\PP\Big(\sum_{k=1}^n B_k\geq \floor{pn}\Big)
\end{equation}
and
\begin{equation}\label{eq:Bern_2}
\PP\big(X^{(\floor{pn})}\leq t\big)=\PP\Big(\sum_{k=1}^n B_k< \floor{pn}\Big).
\end{equation}
\end{lemma}
\begin{proof}[Proof of Lemma~\ref{lem:Bern}]
It suffices to prove eq. (\ref{eq:Bern_1}). 
By the definition of order statistics $X^{(1)}\geq\cdots\geq X^{(n)}$, 
$$
\big\{X^{(\floor{pn})}\geq t\big\}=\big\{\#\{1\leq k\leq n: X_k\geq t\}\geq \floor{pn}\big\}.
$$
The latter event is equivalent to $\big\{\sum_{k=1}^n B_k\geq \floor{pn}\big\}$ where $B_k\sim {\rm Bern}\big(1-F_k(t)\big)$, which leads to eq. (\ref{eq:Bern_1}). 
\end{proof}
By Lemma~\ref{lem:Bern}, the concentration of $X^{(\floor{pn})}$ is translated into the concentration of the sum of independent Bernoulli random variables. Observe, with the definitions in Lemma~\ref{lem:Bern}, that
$$
\EE \sum_{k=1}^n B_k=n\big(1-F_N(t)\big).
$$
We obtain the following concentration inequalities of the $100(1-p)$\%-th percentile $X^{(\floor{pn})}$. 
\begin{theorem}\label{thm:general}
For any $p\in(0,1)$ such that $pn\geq 1$ and $t>0$, denote by
$$
\varphi_F^+(t)=\frac{\floor{pn}}{n}+F_N\big(F_N^{-1}(1-p)+t\big)-1
$$
and
$$
\varphi_F^-(t)=\frac{\floor{pn}}{n}+F_N\big(F_N^{-1}(1-p)-t\big)-1.
$$
Then, for all $t>0$ such that $\varphi_F^+(t)\geq 0$, the following inequality holds
$$
\PP\big(X^{(\floor{pn})}\geq F_N^{-1}(1-p)+t\big)\leq \exp\Big(-2n\cdot\big[\varphi_F^+(t)\big]^2\Big).
$$
Similarly, for all $t>0$, the following inequality holds
$$
\PP\big(X^{(\floor{pn})}\leq F_N^{-1}(1-p)-t\big)\leq \exp\Big(-2n\cdot\big[\varphi_F^-(t)\big]^2\Big).
$$
\end{theorem}
\begin{proof}
By Lemma~\ref{lem:Bern}, we get
$$
\PP\big(X^{(\floor{pn})}\geq F_N^{-1}(1-p)+t\big)=\PP\Big(\sum_{k=1}^n \tilde B_k\geq \floor{pn}\Big)
$$
where $\{\tilde B_k\}_{k=1}^n$ are independent Bernoulli random variables with
$$
\PP\big(\tilde B_k=1\big)=1-F_k\big(F_N^{-1}(1-p)+t\big).
$$
Then, by Hoeffding's inequality (\cite{hoeffding1963probability}), we get 
\begin{align*}
\PP\big(X^{(\floor{pn})}\geq& F_N^{-1}(1-p)+t\big)=\PP\Big(\sum_{k=1}^n \big(\tilde B_k-\EE \tilde B_k\big)\geq \floor{pn}-\sum_{k=1}^n\EE \tilde B_k\Big)\\
=&\PP\Big(\sum_{k=1}^n \big(\tilde B_k-\EE \tilde B_k\big)\geq \floor{pn}+n\cdot F_N\big(F_N^{-1}(1-p)+t\big)-n\Big)\\
=&\PP\Big(\sum_{k=1}^n \big(\tilde B_k-\EE \tilde B_k\big)\geq n\cdot \varphi_F^+(t)\Big)
\leq \exp\Big(-2n\cdot \big[\varphi_F^+(t)\big]^2\Big)
\end{align*}
as long as $\varphi_F^+(t)\geq 0$. Therefore, we conclude with 
$$
\PP\big(X^{(\floor{pn})}\geq F_N^{-1}(1-p)+t\big)\leq \exp\Big(-2n\cdot\big[\varphi_F^+(t)\big]^2\Big).
$$
In an identical fashion, we can prove that
$$
\PP\big(X^{(\floor{pn})}\leq F_N^{-1}(1-p)-t\big)\leq \exp\Big(-2n\cdot\big[\varphi_F^-(t)\big]^2\Big)
$$
as long as $\varphi_F^-(t)\leq 0$.
\end{proof}
Note that we applied Hoeffding's inequality to prove Theorem~\ref{thm:general} which gives neat results. However, Hoeffding's inequality is loose when ${\rm Var}(\tilde B_k)$ is small, namely, when $\PP(\tilde B_k=1)$ is small. As a result, the tail inequalities in Theorem~\ref{thm:general} are sharp only for $p$ around $0.5$. See Remark~\ref{rmk:cauchy} in Section~\ref{sec:cauchy} about the $1$st order statistics of independent Cauchy distribution. 

Theorem~\ref{thm:general} can be further simplified if the cumulative distribution functions admit probability density functions satisfying the regularity conditions. For any $p>0$, we define the binary number $I_p$ as 
\begin{equation}\label{eq:I_p}
I_p=\begin{cases}
0,&\textrm{ if } np \textrm{ is an integer},\\
1,&\textrm{ otherwise}. 
\end{cases}
\end{equation}
\begin{corollary}\label{cor:general}
For any $p\in(0,1)$ such that $pn\geq 1$, we denote by $\tau_{N,p}=F_N^{-1}(1-p)$. Suppose that the cumulative distribution functions $F_k(\cdot)$ are differentiable and $f_k(\cdot)=\frac{d F_k(x)}{dx}$ for all $1\leq k\leq N$. 
For any $t>0$, we define $s_{N,k}(t,p)$ as
\begin{equation}\label{eq:cor_Fk}
s_{N,k}(t,p)=\min_{u\in[\tau_{N,p}-t, \tau_{N,p}+t]} \big|f_k(u) \big|
\end{equation}
Then, for any $0<t$ such that $t\sum_{k=1}^n s_{N,k}(t,p)\geq 2I_p$ with $I_p$ defined in (\ref{eq:I_p}), the following inequalities hold
$$
\PP\big(X^{(\floor{pn})}\geq F_N^{-1}(1-p)+t\big)\leq \exp\bigg(-\Big(\sum_{k=1}^ns_{N,k}(t,p)\Big)^2\cdot\frac{2t^2}{n(I_p+1)^2}\bigg)
$$
and
$$
\PP\big(X^{(\floor{pn})}\leq F_N^{-1}(1-p)-t\big)\leq \exp\bigg(-\Big(\sum_{k=1}^ns_{N,k}(t,p)\Big)^2\cdot \frac{2t^2}{n}\bigg).
$$
\end{corollary}
\begin{proof}[Proof of Corollary~\ref{cor:general}]
By the definition of $\varphi_F^+(\cdot)$ in Theorem~\ref{thm:general}, we immediately obtain
\begin{align*}
\varphi_F^+(t)\geq& p-\frac{I_p}{n}+F_N(\tau_{N,p}+t)-1=F_N(\tau_{N,p}+t)-F_N(\tau_{N,p})-\frac{I_p}{n}\\
=&-\frac{I_p}{n}+\frac{1}{n}\sum_{k=1}^n \big|F_k(\tau_{N,p}+t)-F_k(\tau_{N,p})\big|
=-\frac{I_p}{n}+\frac{1}{n}\sum_{k=1}^n \int_{0}^{t}f_k(u+\tau_{N,p})du.
\end{align*}
By the regularity conditions of $f_k(\cdot)$ as in (\ref{eq:cor_Fk}), we get
$$
\varphi_F^{+}(t)\geq -\frac{I_p}{n}+\frac{t}{n}\cdot \sum_{k=1}^n s_{N,k}(t,p)\geq \frac{t}{n(I_p+1)}\cdot\sum_{k=1}^n s_{N,k}(t,p)
$$
as long as $t\sum_{k=1}^N s_{N,k}(t,p)\geq 2I_p$. Similarly, we get
$$
\varphi_F^-(t)\leq -\frac{t}{n}\cdot \sum_{k=1}^n s_{N,k}(t,p)
$$
which concludes the proof in view of Theorem~\ref{thm:general}. 
\end{proof}

\section{Bounds of median for general scale family of distributions}\label{sec:scale-family}
In this section, we prove the general concentration bounds for the median of independent random variables obeying the scale family distributions. Then, we apply the results to several specific distributions. 
 Suppose that $X_1,\cdots,X_n$ are independent random variables with corresponding cumulative distribution functions $F_k(\cdot)$ for $k=1,\cdots,n$ which belong to the scale family of distributions.

\begin{definition}[\it scale family]\label{def:scale-family}
There exists a continuous cumulative distribution function $D(x)$ defined on $\RR$ and a sequence of positive numbers $\sigma_1,\cdots,\sigma_n$ such that 
$$
F_k(x)=D\Big(\frac{x}{\sigma_k}\Big)\quad \textrm{ for all } k=1,\cdots,n.
$$
\end{definition}

The {\it scale family}  covers many important families of distributions, e.g., {\it normal distribution, Cauchy distribution, uniform distribution, logistic distribution}. In Theorem~\ref{thm:scale-family}, we show that the concentration of the median of independent random variables from {\it scale family} distributions depends on the harmonic mean of the corresponding scale parameters.

\begin{theorem}\label{thm:scale-family}
Suppose that the independent random variables $\{X_k\}_{k=1}^n$ with cumulative distribution functions $\{F_k(\cdot)\}_{k=1}^n$ belong to the scale family as in Definition~\ref{def:scale-family}.
Moreover, assume $D(0)=\frac{1}{2}$ and $D(x)$ admits continuous density function $d(x)$. If $n$ is an even number, then for all $t>0$, we have 
$$
\PP(|M|\geq t)\leq 2\exp\bigg(-\Big(\sum_{k=1}^n\sigma_k^{-1}\Big)^2\cdot \frac{2t^2}{n}\cdot \min_{\substack{|u|\leq t\\ 1\leq k\leq n}}d^2\Big(\frac{u}{\sigma_k}\Big)\bigg)
$$
where $M={\rm Med}\big(X_1,\cdots,X_n\big)$. If $n$ is an odd number, then for all $t>0$ with $t\cdot\min_{|u|\leq t, 1\leq k\leq n}d(u/\sigma_k)\cdot \sum_{k=1}^n\sigma_k^{-1}\geq 2$, we have 
$$
\PP(|M|\geq t)\leq 2\exp\bigg(-\Big(\sum_{k=1}^n\sigma_k^{-1}\Big)^2\cdot \frac{t^2}{2n}\cdot \min_{\substack{|u|\leq t\\ 1\leq k\leq n}}d^2\Big(\frac{u}{\sigma_k}\Big)\bigg).
$$
\end{theorem}

\begin{proof}[Proof of Theorem~\ref{thm:scale-family}]
Let $p=1/2$. 
Following the notations in Corollary~\ref{cor:general}, we can write the probability density function $f_k(x)$ as 
$$
f_k(x)=\frac{1}{\sigma_k}\cdot d\Big(\frac{x}{\sigma_k}\Big),\quad 1\leq k\leq n.
$$
Because $D(0)=\frac{1}{2}$, we get $\tau_{N,p}=0$ as defined in Corollary~\ref{cor:general}. Then, for any $t>0$, we get 
\begin{align*}
s_{N,k}(t,p)=\frac{1}{\sigma_k}\cdot \min_{|u|\leq t}d\Big(\frac{u}{\sigma_k}\Big)
\end{align*}
and then
\begin{align*}
\sum_{k=1}^n s_{N,k}(t,p)\geq \min_{\substack{|u|\leq t\\ 1\leq k\leq n}}d\Big(\frac{u}{\sigma_k}\Big)\cdot \sum_{k=1}^n \sigma_k^{-1}.
\end{align*}
If $n$ is an even number, we immediately get, by Corollary~\ref{cor:general}, that 
$$
\PP(|M|\geq t)\leq 2\exp\bigg(-\Big(\sum_{k=1}^n\sigma_k^{-1}\Big)^2\cdot \frac{2t^2}{n}\cdot \min_{\substack{|u|\leq t\\ 1\leq k\leq n}}d^2\Big(\frac{u}{\sigma_k}\Big)\bigg)
$$
If $n$ is an odd number and $t\min_{\substack{|u|\leq t\\ 1\leq k\leq n}}d(u/\sigma_k)\cdot \sum_{k=1}^n\sigma_k^{-1}\geq 2$, then, by Corollary~\ref{cor:general}, we get that 
$$
\PP(|M|\geq t)\leq 2\exp\bigg(-\Big(\sum_{k=1}^n\sigma_k^{-1}\Big)^2\cdot \frac{t^2}{2n}\cdot \min_{\substack{|u|\leq t\\ 1\leq k\leq n}}d^2\Big(\frac{u}{\sigma_k}\Big)\bigg).
$$
\end{proof}
Now we apply Theorem~\ref{thm:scale-family} to some important distributions, including heavy tailed distributions, to present the explicit concentration bounds for the medians of those corresponding random variables. 

\subsection{Median of independent normal random variables}\label{sec:normal}
We first develop bounds for the median of independent non-identical centered normal random variables. Let $X_k\sim\calN(0,\sigma_k^2)$ be independent normal random variables for $k=1,\cdots,n$. Similarly as above, we denote their median by
$$
M={\rm Med}\big(X_1,\cdots,X_n\big).
$$
As introduced in Section~\ref{sec:intro}, in the case of identical distributions with $\sigma_k\equiv \sigma$, it is well known that the median $|M|=O_P\big(\frac{\sigma}{n^{1/2}}\big)$. For general $\sigma_k$'s, the upper bound of $|M|$ is characterized in the following Corollary.

\begin{corollary}\label{cor:normal}
Let $X_k\sim\calN(0,\sigma_k^2)$ for $k=1,\cdots,n$ be independent. Let $t>0$ such that
\begin{equation}\label{eq:normal_cond}
 \frac{n^{1/2}t^{1/2}}{0.35\sum_{k=1}^n \sigma_k^{-1}}\leq \frac{1}{2}\min_{1\leq k\leq n}\sigma_k.
\end{equation}
If $n$ is an even number, then,  with probability at least $1-2e^{-2t}$,
$$
\big|M \big|\leq  \frac{n^{1/2}t^{1/2}}{0.35\sum_{k=1}^n \sigma_k^{-1}}
$$
where $M={\rm Med}\big(X_1,\cdots,X_n\big)$. On the other hand, if $n$ is an odd number and $tn\geq 4$, then with probability at least $1-2e^{-t/2}$,
$$
\big|M \big|\leq  \frac{n^{1/2}t^{1/2}}{0.35\sum_{k=1}^n \sigma_k^{-1}}.
$$
\end{corollary}
\begin{proof}[Proof of Corollary~\ref{cor:normal}]
We will directly apply Theorem~\ref{thm:scale-family} to prove this corollary. We set the $t$ in Theorem~\ref{thm:scale-family} as 
$$
t=\frac{n^{1/2}\cdot s^{1/2}}{0.35\sum_{k=1}^n\sigma_k^{-1}}
$$ 
for some $s>0$. Now, we assume that $t\leq min_{1\leq k\leq n}\sigma_k/2$. Then, following the notations in Theorem~\ref{thm:scale-family}, we get
\begin{align*}
\min_{\substack{|u|\leq t\\ 1\leq k\leq n}} d\Big(\frac{u}{\sigma_k}\Big)=d\Big(\frac{t}{\min_{1\leq k\leq n}\sigma_k}\Big)=d\Big(\frac{1}{2}\Big)
\end{align*}
where $d(\cdot)$ is actually the p.m.f. of the standard normal distribution. Then, we get 
$$
\min_{\substack{|u|\leq t\\ 1\leq k\leq n}} d\Big(\frac{u}{\sigma_k}\Big)\geq \frac{1}{\sqrt{2\pi }}e^{-1/8}\geq 0.35.
$$
If $n$ is an even number, by Theorem~\ref{thm:scale-family}, we immediately obtain
$$
\PP\bigg(|M|\geq \frac{n^{1/2}s^{1/2}}{0.35\sum_{k=1}^n\sigma_k^{-1}}\bigg)\leq 2\exp(-2s)
$$
If $n$ is an odd number and $ns\geq 4$, then by Theorem~\ref{thm:scale-family}, we get
$$
\PP\bigg(|M|\geq \frac{n^{1/2}s^{1/2}}{0.35\sum_{k=1}^n\sigma_k^{-1}}\bigg)\leq 2\exp(-s/2).
$$
Therefore, we conclude the proof by changing $s$ to $t$.  
\end{proof}
\begin{remark}
Recall that the sample mean $\bar{X}=n^{-1}\sum_{k=1}^n X_k$ has normal distribution with zero mean and variance $n^{-2}\sum_{k=1}^n\sigma_k^2$ implying that (by Z table)
$$
\PP\bigg(\big|\bar{X}\big|\geq \frac{\big(\sum_{k=1}^n\sigma_k^2\big)^{1/2}}{n}\bigg)\geq 0.30.
$$
By the famous harmonic mean inequality:
$$
\frac{n}{\sum_{k=1}^n \sigma_k^{-1}}\leq \bigg(\frac{\sum_{k=1}^n\sigma_k^2}{n}\bigg)^{1/2}
$$
, we conclude that  $\EE\big|{\rm Med}\big(X_1,\cdots,X_n\big)\big|\lesssim \EE\big|{\rm Mean}\big(X_1,\cdots,X_n\big)\big|$.
\end{remark}

\begin{remark}
Clearly, $\sum_{k=1}^n \sigma_k^{-1}\geq n\big(\max_{1\leq k\leq n} \sigma_k\big)^{-1}$.  Therefore, Corollary~\ref{cor:normal} immediately imply the conservative bound
$$
\PP\bigg(\big| M\big|\geq 2.86\sqrt{\frac{t}{n}}\cdot \max_{1\leq k\leq n} \sigma_k\bigg)\leq 2\exp\{-2t\}
$$
as long as condition holds as in eq. (\ref{eq:normal_cond}). It is intuitively understood as the median of a group of small-variance random variables is often smaller than the median of a group of large-variance random variables.
\end{remark}

\begin{remark}
We note that the constant on the right hand side of eq. (\ref{eq:normal_cond}) is unnecessarily to be $\frac{1}{2}$. In fact, any positive bounded constant suffices to take the same role in which case the constant in the later claim should be adjusted accordingly.
\end{remark}

\begin{remark}
Recall the non-dispersion condition in eq. (\ref{eq:normal_cond}) which requires $n^{1/2}\lesssim \Big(\sum_{k=1}^n\sigma_k^{-1}\Big)\cdot \min_{1\leq k\leq n} \sigma_k$.  Basically, it requires that $\{\sigma_k\}_{k=1}^n$ shall not be too dispersed. Without loss of generality, assume $\sigma_1\leq \cdots\leq \sigma_n$ with $\sigma_1=1$. Moreover, suppose that $\sigma_k$ grows with $k$ like the order $\sigma_k\asymp k^{\alpha}$ for $\alpha>0$. Then, 
$$
\Big(\sum_{k=1}^n\sigma_k^{-1}\Big)\cdot \min_{1\leq k\leq n} \sigma_k\asymp \sum_{k=1}^n k^{-\alpha}.
$$
In order to guarantee the condition in eq. (\ref{eq:normal_cond}), it suffices to require $\alpha\leq \frac{1}{2}$, i.e., $\sigma_k$ shall not grow faster than $k^{1/2}$. 
\end{remark}

\begin{remark}
Let's compare with the existing results of the concentration inequalities for order statistics in the literature. The concentration inequality for the median of i.i.d. standard Gaussian random variables in \cite[Proposition~4.6]{boucheron2012concentration} shows that, when $n$ is even, for all $t>0$, 
$$
\PP\Big(M-\EE M > 4\sqrt{t/n\log 2}+4t\sqrt{2/n^2\log 2}\Big)\leq e^{-t}
$$
where $M$ is the median. It implies a sub-gaussian tail for $t\leq n/2$ and a sub-exponential tail for $t>n/2$. The sub-gaussian tail we obtained in Corollary~\ref{cor:normal} is due to the condition (\ref{eq:normal_cond}) which is equivalent to $t=O(n)$. 
\end{remark}

\subsection{Median of independent Cauchy random variables}\label{sec:cauchy}
Let $\{X_k\}_{k=1}^n$ be independent Cauchy random variables and for each $X_k$, its probability density function is given as
\begin{equation}\label{eq:cauchy_fk}
f_k(x)=\frac{1}{\pi \sigma_k}\cdot \frac{\sigma_k^2}{x^2+\sigma_k^2}\quad \forall\ x\in \RR
\end{equation}
for $\sigma_k>0$. The corresponding cumulative distribution function is written as
$$
F_k(x)=\int_{-\infty}^x f_k(t)dt=\frac{1}{\pi}\cdot {\rm arctan}\Big(\frac{x}{\sigma_k}\Big)+\frac{1}{2}
$$
implying that $F_k(0)=\frac{1}{2}$. Clearly, $\EE X_k$ and ${\rm Var}(X_k)$ are undefined, but ${\rm Median}(X_k)=0$ for all $1\leq k\leq n$. Since $\{F_k(x)\}_{k=1}^n$ belong to the scale family as in Definition~\ref{def:scale-family} with function 
$$
D(x)=\frac{1}{\pi}\cdot {\rm arctan}(x)+\frac{1}{2},
$$
we obtain its density function $d(u)=\frac{1}{\pi(x^2+1)}$ concluding that $\min_{|u|\leq 1/2}d(u)=\frac{4}{5\pi}$. By the proof of Corollary~\ref{cor:normal}, we immediately obtain the following concentration bound for the median of independent Cauchy random variables. 
\begin{corollary}\label{cor:cauchy}
Let $\{X_k\}_{k=1}^n$ be independent Cauchy random variables with corresponding probability density functions as in eq. (\ref{eq:cauchy_fk}). Let $t>0$ such that
$$
\frac{5\pi n^{1/2}t^{1/2}}{4\sum_{k=1}^n\sigma_k^{-1}}\leq \frac{1}{2}\min_{1\leq k\leq n}\sigma_k.
$$
If $n$ is an even number, then we get 
$$
\PP\bigg(|M|\geq \frac{5\pi n^{1/2}t^{1/2}}{4\sum_{k=1}^n\sigma_k^{-1}}\bigg)\leq 2\exp\{-2t\}
$$
where $M={\rm Med}\big(X_1,\cdots,X_n\big)$. On the other hand, if $n$ is an odd number and $nt\geq 4$, then we have 
$$
\PP\bigg(|M|\geq \frac{5\pi n^{1/2}t^{1/2}}{4\sum_{k=1}^n\sigma_k^{-1}}\bigg)\leq 2\exp\{-t/2\}. 
$$
\end{corollary}
\begin{remark}\label{rmk:cauchy}
One possible misconception of the concentration inequalities in Theorem~\ref{thm:general} is that all the percentiles have sub-gaussian tails. This is actually not true. The tails depend on the function $\varphi_F^+(t)$. For instance, let us consider the $1$st order statistics ($p=\frac{1}{n}$) of i.i.d. Cauchy random variables with $\sigma_k\equiv 1$.  In this case, we have $F_N(x)=\frac{1}{2}+\frac{1}{\pi}\cdot {\rm arctan}(x)$ and $F_N'(x)=\frac{1}{\pi(1+x^2)}$ which is a convex function for $x\geq 1$.
As a result, for any $t>0$ and $n\geq 6$, we get 
\begin{align*}
\varphi_F^+(t)=&F_N\Big(\tan\big[\pi\big(\frac{1}{2}-\frac{1}{n}\big)\big]+t\Big)-F_N\Big(\tan\big[\pi\big(\frac{1}{2}-\frac{1}{n}\big)\big]\Big)\\
\leq& \frac{1}{\pi}\cdot \frac{t}{1+\big[t/2+\tan\big(\pi(0.5-n^{-1})\big)\big]^2}\\
\approx& \frac{1}{\pi}\cdot \frac{t}{1+\big[t/2+\cos^{-1}\big(\pi(0.5-n^{-1})\big)\big]^2}\approx \frac{1}{\pi}\cdot \frac{t}{1+(t/2+\pi^{-1}n)^2}.
\end{align*}
Therefore, for large $t>n$, the tail of the $1$st order statistics by Theorem~\ref{thm:general} is $e^{-O(n/t^2)}$.  This tail bound is loose since it is famous that the $1$st order statistics of i.i.d. Cauchy random variables follows the inverse exponential distribution for large $n$. 
\end{remark}

\subsection{Median of independent Laplace random variables}
Let $\{X_k\}_{k=1}^n$ be independent Laplace random variables and for each $X_k$, its probability density function is given as
\begin{equation}\label{eq:laplace_fk}
f_k(x)=\frac{1}{2\sigma_k}\exp\bigg\{-\frac{|x|}{\sigma_k} \bigg\}\quad \forall\ x\in \RR
\end{equation}
for $\sigma_k>0$. The corresponding cumulative distribution function is written as
$$
F_k(x)=
\begin{cases}
\frac{1}{2}\exp\Big\{\frac{x}{\sigma_k}\Big\}\quad& \textrm{if } x\leq 0\\
1-\frac{1}{2}\exp\Big\{-\frac{x}{\sigma_k}\Big\}\quad& \textrm{if } x\geq 0
\end{cases}
$$
implying that $F_k(0)=\frac{1}{2}$.  Since $\{F_k(x)\}_{k=1}^n$ belong to the scale family as in Definition~\ref{def:scale-family} with function 
$$
D(x)=\begin{cases}
\frac{1}{2}\exp\{x\}\quad& \textrm{if } x\leq 0\\
1-\frac{1}{2}\exp\{-x\}\quad& \textrm{if } x\geq 0
\end{cases}
$$
we obtain its density function $d(u)=\frac{1}{2}\exp\{-|x|\}$ concluding that $\min_{|u|\leq 1/2}d(u)=\frac{1}{2\sqrt{e}}$. By the proof of Corollary~\ref{cor:normal}, we immediately obtain the following concentration bound for the median of Laplace random variables. 
\begin{corollary}\label{cor:cauchy}
Let $\{X_k\}_{k=1}^n$ be independent Laplace random variables with corresponding probability density functions as in eq. (\ref{eq:laplace_fk}). Let $t>0$ such that
$$
\frac{2\sqrt{e} n^{1/2}t^{1/2}}{\sum_{k=1}^n\sigma_k^{-1}}\leq \frac{1}{2}\min_{1\leq k\leq n}\sigma_k.
$$
If $n$ is an even number, then we obtain
$$
\PP\bigg(|M|\geq \frac{2\sqrt{e} n^{1/2}t^{1/2}}{\sum_{k=1}^n\sigma_k^{-1}}\bigg)\leq 2\exp\{-2t\}
$$
where $M={\rm Med}\big(X_1,\cdots,X_n\big)$. On the other hand, if $n$ is an odd number and $nt\geq 4$, then we get
$$
\PP\bigg(|M|\geq \frac{2\sqrt{e} n^{1/2}t^{1/2}}{\sum_{k=1}^n\sigma_k^{-1}}\bigg)\leq 2\exp\{-t/2\}.
$$
\end{corollary}

\section{Other percentiles of independent normal random variables.}\label{sec:normal_general}
In Section~\ref{sec:normal}, we developed the concentration bounds for the median of independent centered normal random variables. Occasionally, other percentiles than median are of interest. Still, let $X_k\sim\calN(0,\sigma_k^2)$ be independent and we are interested in its percentiles $X^{(\floor{(0.5+\tau)n})}$ and $X^{(\floor{(0.5-\tau)n})}$ for some small number $\tau\leq \frac{1}{4}$.  Recall, by borrowing the notations from Theorem~\ref{thm:general}, that
$$
F_N(x)=\frac{1}{n}\sum_{k=1}^n \Phi\Big(\frac{x}{\sigma_k}\Big)
$$
where $\Phi(\cdot)$ represents the cumulative distribution function of standard normal random variables.

\begin{theorem}\label{cor:normal_general}
Let $\tau\in[0, 0.25]$. 
For any $t\geq n^{-1}$ such that 
\begin{equation}\label{eq:normal_general_cond1}
\frac{\sqrt{8tn\pi}e^{1/4}}{\sum_{k=1}^n\sigma_k^{-1}\exp\big\{-[F_N^{-1}(0.5+\tau)]^2/\sigma_k^2\big\}}\leq \min_{1\leq k\leq n} \sigma_k/2
\end{equation}
, we have
$$
\PP\Big(X^{(\floor{(0.5-\tau)n})}\geq F_N^{-1}(0.5+\tau)+\frac{\sqrt{8tn\pi}e^{1/4}}{\sum_{k=1}^n\sigma_k^{-1}\exp\big\{-[F_N^{-1}(0.5+\tau)]^2/\sigma_k^2\big\}}\Big)\leq 2e^{-2t}.
$$
Similarly, For any $t\geq 0$ such that 
$$
\frac{\sqrt{8tn\pi}e^{1/4}}{\sum_{k=1}^n\sigma_k^{-1}\exp\big\{-[F_N^{-1}(0.5-\tau)]^2/\sigma_k^2\big\}}\leq \min_k \sigma_k/2
$$
, we have
$$
\PP\Big(X^{(\floor{(0.5+\tau)n})}\leq F_N^{-1}(0.5-\tau)-\frac{\sqrt{8tn\pi}e^{1/4}}{\sum_{k=1}^n\sigma_k^{-1}\exp\big\{-[F_N^{-1}(0.5-\tau)]^2/\sigma_k^2\big\}}\Big)\leq 2e^{-2t}.
$$
\end{theorem}
\begin{proof}[Proof of Theorem~\ref{cor:normal_general}]
We only prove the lower bound of $X^{(\floor{(0.5-\tau)n})}$ since the proof of upper bound for $X^{(\floor{(0.5+\tau)n})}$ is similar. 
 Since $\Phi(\cdot)$ is Lipschitz on $\RR$ and for any $x,y\geq 0$,
 $$
\frac{y}{\sqrt{2\pi}}\cdot e^{-(x+y)^2/2}\leq \Phi(x+y)-\Phi(x)=\frac{1}{\sqrt{2\pi}}\int_x^{x+y}e^{-t^2/2}dt\leq \frac{y}{\sqrt{2\pi}}\cdot e^{-x^2/2}. 
 $$
 Therefore, for $0<t\leq \frac{1}{2}\min_{1\leq k\leq n} \sigma_k$, we have
\begin{align*}
\varphi_F^+(t)=&\frac{\floor{(0.5-\tau)n}}{n}+F_N\big(F_N^{-1}(0.5+\tau)+t\big)-1\\
\geq& -\frac{1}{n}+\frac{1}{n}\sum_{k=1}^n \bigg(F_k\Big(F_N^{-1}(0.5+\tau)+t\Big)-F_k\Big(F_N^{-1}(0.5+\tau)\Big)\bigg)\\
\geq&\frac{1}{n}\sum_{k=1}^n \frac{t}{\sqrt{2\pi}\sigma_k}\exp\Big\{-\frac{\big[F_N^{-1}(0.5+\tau)\big]^2}{\sigma_k^2}-\frac{t^2}{\sigma_k^2}\Big\}-\frac{1}{n}\\
\geq&\frac{e^{-1/4}}{n}\sum_{k=1}^n \frac{t}{\sqrt{2\pi}\sigma_k}\exp\Big\{-\frac{\big[F_N^{-1}(0.5+\tau)\big]^2}{\sigma_k^2}\Big\}-\frac{1}{n}\\
\geq& \frac{e^{-1/4}}{2n}\sum_{k=1}^n \frac{t}{\sqrt{2\pi}\sigma_k}\exp\Big\{-\frac{\big[F_N^{-1}(0.5+\tau)\big]^2}{\sigma_k^2}\Big\}
\end{align*}
where the last inequality holds whenever
$$
e^{-1/4}\sum_{k=1}^n \frac{t}{\sqrt{2\pi}\sigma_k}\exp\Big\{-\frac{\big[F_N^{-1}(0.5+\tau)\big]^2}{\sigma_k^2}\Big\}\geq 2.
$$
By Theorem~\ref{thm:general}, we conclude that 
\begin{align*}
\PP\Big(X^{(\floor{(0.5-\tau)n})}&\geq F_N^{-1}(0.5+\tau)+t \Big)\\
\leq& 2\exp\bigg\{-\frac{e^{-1/2}t^2}{4n\pi}\Big(\sum_{k=1}^n\sigma_k^{-1}\exp\big\{-\big[F_N^{-1}(0.5+\tau)\big]^2/\sigma_k^2\big\}\Big)^2\bigg\}.
\end{align*}
By setting $t=\frac{\sqrt{8un\pi}e^{1/4}}{\sum_{k=1}^n\sigma_k^{-1}\exp\big\{-[F_N^{-1}(0.5+\tau)]^2/\sigma_k^2\big\}}$ for any $u\geq 1$ such that $t\leq \min_{1\leq k\leq n} \sigma_k/2$, we conclude with 
$$
\PP\Big(X^{(\floor{(0.5-\tau)n})}\geq F_N^{-1}(0.5+\tau)+\frac{\sqrt{8un\pi}e^{1/4}}{\sum_{k=1}^n\sigma_k^{-1}\exp\big\{-[F_N^{-1}(0.5+\tau)]^2/\sigma_k^2\big\}}\Big)\leq 2e^{-2u}.
$$
\end{proof}

\section{Discussion}
In this note, we develop the concentration bounds for the sample percentiles of independent but non-identical random variables. For a wide class of symmetric distributions, we show that the sample median has a magnitude related with the harmonic mean of the associated standard deviations or scale parameters.  There are several important unsolved questions for further investigation. The first one is on the lower bound of the expectation of the median's magnitude. Only upper bound is proved theoretically in this note. It is, however, unclear whether the expected magnitude of the median is indeed of the order related with the harmonic mean, not even for the normal random variables. 
Another interesting question is its generalization to the tail bound for the order statistics of dependent zero-mean random variables. In \cite[Theorem~3]{latala2010order}, the authors proved a sub-exponential tail bound for the order statistics of log-concave random variables (could be dependent) with unit variance and zero mean. We are wondering whether the median of dependent centered normal random variables is also characterized by the  harmonic mean of the individual standard deviations. 

\bibliographystyle{abbrv}
\bibliography{refer}
\end{document}